\documentclass[12pt, reqno]{amsart}
\usepackage{amsmath, amsthm, amscd, amsfonts, amssymb, graphicx, color}
\usepackage{enumerate}
\usepackage[bookmarksnumbered, colorlinks, plainpages]{hyperref}
\hypersetup{colorlinks=true,linkcolor=red, anchorcolor=green, citecolor=cyan, urlcolor=red, filecolor=magenta, pdftoolbar=true}

\newcommand{\mcal}{\mathcal}
\newcommand{\q}{\left\{}
\newcommand{\w}{\right\}}
\newcommand{\re}{\mathbb{R}}
\newcommand{\cn}{\mathbb{C}}

\newcommand{\tb}{\textbf}

\newcommand{\ti}{\textit}

\newcommand{\bt}{\beta}

\newtheorem{theorem}{Theorem}[section]
\newtheorem{lemma}[theorem]{Lemma}
\newtheorem{proposition}[theorem]{Proposition}

\theoremstyle{definition}

\theoremstyle{remark}
\newtheorem{remark}[theorem]{Remark}
\numberwithin{equation}{section}

\newcommand{\opn}{\operatorname}
\newcommand{\rank}{\opn{rank}\;}

\newcommand{\supp}{\opn{supp}\;}

\newcommand{\av}{\mathcal{V}}

\newcommand{\V}{\mathcal V}

\begin{document}

\setcounter{page}{1}

\title[The Nonsingular Cubic Binary Moment Problem]{A New Approach to the Nonsingular Cubic Binary Moment Problem}

\author[R. Curto, \MakeLowercase{and} S. Yoon]{Ra\'ul E. Curto,$^1$$^{*}$ \MakeLowercase{and} Seonguk Yoo$^2$}

\address{$^{1}$Department of Mathematics, The University of Iowa, Iowa City, Iowa 52242, USA.}
\email{\textcolor[rgb]{0.00,0.00,0.84}{raul-curto@uiowa.edu}}

\address{$^{2}$Department of Mathematics, Sungkyunkwan University, Suwon, 16419, Korea.}
\email{\textcolor[rgb]{0.00,0.00,0.84}{seyoo73@gmail.com}}


\let\thefootnote\relax\footnote{Copyright 2016 by the Tusi Mathematical Research Group.}



\date{Received: xxxxxx; Revised: yyyyyy; Accepted: zzzzzz.
\newline \indent $^{*}$Corresponding author}

\begin{abstract}
We present an alternative solution to nonsingular cubic moment problems, using techniques that are expected to be useful for higher-degree truncated moment problems. \ In particular, we apply the theory of recursively determinate moment matrices to deal with a case of rank-increasing moment matrix extensions.      
\end{abstract} \maketitle

\section{Introduction}\label{intro}

Given a doubly indexed finite sequence of real numbers $\beta \equiv\beta^{(m)}= \{  \beta_{00},$ $\beta_{10},$ \newline $\beta_{01},\cdots, \beta_{m,0},\beta_{m-1,1},\cdots, \beta_{1,m-1},\beta_{0,m} \}$ with $\beta_{00}>0$, the \ti{truncated real moment problem} (TRMP) entails seeking necessary and sufficient conditions for the existence of  a positive Borel measure $\mu$ supported in the real plane $\mathbb{R}^2$ such that
\begin{eqnarray*}
\beta_{ij}=\int x^i y^j \,\, d\mu \,\,\,(i,j\in \mathbb Z_+, \ 0\leq i+j\leq m).\ 
\end{eqnarray*}
When such a measure exists, we say that $\mu$ is a \ti{representing measure} for $\beta$ and that the TRMP is \ti{soluble}.  

There is a parallel \ti{truncated complex moment problem} (TCMP), for a finite sequence of complex numbers $\gamma \equiv \gamma ^{(m)}:\gamma 
_{00},\gamma _{01},\gamma _{10},\cdots ,$ $\gamma _{0,m},$\ $\gamma
_{1,m-1},\cdots ,\gamma _{m-1,1},$ $\gamma _{m,0},$ with $\gamma _{00}>0$
and $\gamma _{ji}=\bar{\gamma}_{ij}.$ \ Here TCMP consists of finding a positive Borel measure $\mu $
supported in the complex plane $\mathbb{C}$ such that $\gamma _{ij}=\int
\bar{z}^{i}z^{j}\;d\mu \;\;\;(0\leq i+j\leq m).\ $ 
It is well-known that TRMP are TCMP are equivalent for an even integer $m$ (cf. \cite[Proposition 1.12]{tcmp6}), and hence any techniques developed for TCMP are transferable to TRMP. \  Both problems are simply referred to as the \ti{truncated moment problem}  (TMP). \  

In a series of papers, for the case when $m=2d$, the first named author and L.A. Fialkow found solutions for various truncated moment problems; for instance, they obtained complete solutions for $m=2$ and $m=4$ (cf. \cite{tcmp1, tcmp6, FiNi10, CuYoo2}). \ Some  solutions are based on matrix positivity and extension, combined with a ``functional calculus'' (to be briefly discussed in Section \ref{s-prel}) for the columns of the associated moment matrix. \ This matrix is defined as follows. \ For a real moment sequence $\beta^{(2d)}$ of even degree, the \textit{moment matrix} $\mcal M(d) \equiv\mathcal{M}(d)(\beta^{(2d)})$ is given by
$$
\mathcal{M}(d)(\beta^{(2d)}):=(\beta_{\textbf{ i} +\textbf{j} })_{\textbf{ i}, \, \textbf{j}\in \mathbb Z^2_+: |\textbf{i}|,\, |\textbf{j}|\leq 2d}.\ 
$$
If we label the columns of $\mathcal M{(d)}$ with the degree lexicographical order, $\emph{1},X,Y,X^2,$ $X Y,Y^2, \newline \cdots,X^d,\cdots,Y^d$, we can then use the functional calculus for columns of $\mathcal{M}(d)$, introduced in \cite{tcmp1}. \ The moment matrix $\mathcal{M}(d)$ is Hankel by rectangular blocks; for instance, 
\begin{equation}
\mathcal{M}(2) \equiv 
\begin{pmatrix}
\beta{00} & | &\beta{10}& \beta{01} & | &\beta{20} & \beta{11} & \beta{02} \\
--&-&--&--&-&--&--&-- \\ 
\beta{10} & | &\beta{20}& \beta{11} & | &\beta{30} & \beta{21} & \beta{12} \\
\beta{01} & | &\beta{11}& \beta{02} & | &\beta{21} & \beta{12} & \beta{03} \\
--&-&--&--&-&--&--&-- \\
\beta{20} & | &\beta{30}& \beta{21} & | &\beta{40} & \beta{31} & \beta{22} \\
\beta{11} & | &\beta{21}& \beta{12} & | &\beta{31} & \beta{22} & \beta{13} \\
\beta{02} & | &\beta{12}& \beta{03} & | &\beta{22} & \beta{13} & \beta{04}
\end{pmatrix}.
\end{equation}
   
When $m = 2d + 1$, a general solution to partial cases of the TMP can be found in  \cite{Kim11} and \cite{KiWo11} as well as a solution to the \ti{truncated matrix moment problem}; a solution to the \ti{cubic complex moment problem} (when $m=3$)  was given in \cite{Kim11}. \ We know from \cite[Proposition 1.12]{tcmp6} that the complex and real truncated moment problems are equivalent, in the sense that there exists a bridging map that allows one to translate the hypotheses and conclusions for TCMP into similar hypotheses and conclusions for TRMP, and vice versa. \  Our approach here, however, focuses on the real cubic moment problem, and analyzes it in its own right.

In this note we consider  cubic \emph{real} moment problems and present an alternative solution to the ``nonsingular'' case (i.e., $\mathcal{M}(1)$ invertible; see Section \ref{s-cmp} for the formal definition). \ Our idea is to extend the initial data $\beta^{(3)}$ to an even-degree $\beta^{(4)}$, for which the associated moment matrix $\mathcal{M}(2)$ has rank $3$ or $4$. \ We then prove that $\mathcal{M}(2)$ is a flat extension of $\mathcal{M}(1)$, or is a flat extension of a $4\times4$ submatrix, or admits a flat extension $\mathcal{M}(3)$. \ In all three cases we find a finitely atomic representing measure for $\beta^{(3)}$.

We anticipate that the present work will contribute to our understanding of higher-degree moment problems, beginning with the quintic moment problem. \ We also expect that solutions to odd-degree moment problems will be applied to solve the subnormal completion problem studied in \cite{CLY}. \


\section{Preliminaries} \label{s-prel}

When we build a moment matrix $\mcal M(2)$ out of a cubic finite sequence, the lower right-hand $3\times3$ block will include all quartic moments, which will need to remain undefined. \ To obtain our main results, we will choose appropriate quartic moments and show that $\mcal M(2)$ has a representing measure. \ In order to describe this process in detail, we need to review basic TMP notation and results pertaining to the even-degree case. \  

\subsection*{\textbf{Necessary Conditions}\label{NCS}}
In order to discuss basic necessary conditions for the existence of a measure, let $\mu$ be a representing measure of the  even-degree moment sequence $\beta\equiv\beta^{(2d)}$. \  First, we recall that
\begin{eqnarray*}
0 \leq \int| p(x,y)|^2 \, d\mu
 = \sum_{i,j,k,l} a_{ij}a_{kl} \int x^{i+l} y^{j+k} \, d\mu
 = \sum_{i,j,k,l} a_{ij} a_{kl} \beta_{i+l} \beta_{j+k}
\end{eqnarray*}
if and only if $\mathcal{M}(d) \geq 0$. \ 

Let $\mcal P_k$ denote the set of bivariate polynomials in $\re[x,y]$ whose degree is at most $k$ and let  $\mcal C_{\mcal M(d)}$ denote the column space of $\mcal M(d)$. \  For $k \le d$, we now define an assignment from $\mathcal P_k$ to $\mathcal C_{\mathcal{M}(d)}$; given a polynomial $p(x,y) \equiv \sum_{ij}a_{ij}x^{i}y^{j}$, we let $p(X,Y):=\sum_{ij}a_{ij}X^{i}Y^{j}$ (so that $p(X,Y
)\in \mathcal{C}_{\mathcal{M}(d)}$), which defines the above mentioned functional calculus. \ 
We also let $\mathcal{Z}(p)$ denote the zero set of $p$ and define the \textit{algebraic variety} of $\beta $ by
\begin{equation} \label{variety}
\mathcal{V} \equiv \mathcal{V}(\beta ):=\bigcap {}_{p\, (x,y)=0,\, \deg \, p\, \leq d}\; \mathcal{Z}(p). \
\end{equation}
If $\widehat{p}$ denotes the column vector of coefficients of $p$, then we know that $p(X,Y)=\mathcal{M}(d)\widehat{p}$; as a consequence, $p(X,Y)=\mathbf{0}$ if and only if $\widehat{p} \in \opn{ker}\; \mathcal{M}(d)$. \  
Another  necessary condition we will use is $\opn{supp}\; \mu \subseteq \mathcal{V}(\beta )$ and
$r:=\opn{rank}\; \mathcal{M}(d)\leq \opn{card}\;\opn{supp}\;\mu \leq v:= \opn{card}\V$; this condition is called the \textit{variety condition} \cite{tcmp1}. \ 
In addition, if $p$ is any polynomial of degree at most $2d$ such that $p|_{\mathcal{V}}\equiv 0$, then the \textit{Riesz functional} $\Lambda $ must satisfy $
\Lambda (p):=\int p \, d\mu =0$, which is referred to as \textit{Consistency} of the moment sequence. \ 
The main results in \cite{tcmp11} state that the above mentioned conditions together with Consistency are sufficient for solubility in the {extremal} case ($r=v$); moreover, in \cite{tcmp11} the authors showed that Consistency cannot be replaced by the weaker condition that $\mathcal{M}(d)$ is \textit{recursively generated}, (RG), that is, $p(X,Y)=\mathbf{0} \Rightarrow (pq)(X,Y)=\mathbf{0}$, for each polynomial $q$ with $\deg (pq)\leq d$. \ 

In summary, positive semidefiniteness alone is sufficient to solve the quadratic moment problem ($d=1$) \cite{tcmp1}. \ However, for $d>1$, the solubility of TMP requires more. \ For instance, the solution of the quartic moment problem ($d=2$) requires positive semidefiniteness, the variety condition, and the (RG) property (this last property requires that the moment matrix be recursively generated; cf. \cite{tcmp1}, \cite{tcmp6}, \cite{FiNi10}). \

\subsection*{Flat Extensions}

We recall that $\mathcal{M}(d)$ is said to be \ti{flat} if $\rank \mcal M(d)=\rank \mcal M(d-1)$, that is, $\mcal M(d)$ is a rank-preserving positive extension of $\mathcal M(d-1)$; in this case, $\mcal M(d)$ has a unique  $\opn{rank}\; \mathcal{M}(d)$-atomic measure. \  
Furthermore, it is known that if $\mathcal{M}(d)$ has a  positive  extension $\mathcal M(d+k)$ for some $k\in \mathbb Z_+$, which in turn admits a flat extension  $\mathcal M(d+k+1)$, then $\beta$ has a $\opn{rank}\; \mathcal{M}(d+k)$-atomic measure \cite[Theorem 1.5]{tcmp3}. \  
This result is referred to as the Flat Extension Theorem; it is probably the most efficient, concrete solution to TMP, even though the construction of an extension is usually  difficult for a high-degree TMP. \
  
We will use the Flat Extension Theorem in the proof of our main result; thus, we need to briefly describe the process for building a flat extension. \ Since a moment matrix extension $\mathcal{M}(d+1)$ of $\mathcal{M}(d)$ can be written as $\mathcal M(d+1)=\left(\begin{array}{cccc} \mathcal{M}(d) & B(d+1) \\ B(d+1)^{*} & C(d+1) \end{array} \right)$, for some rectangular matrices $B(d+1)$ and $C(d+1)$, we can adapt a classical result given by J.L. Smul'jan in the search for a positive $\mathcal M(d+1)$.\ 

\begin{theorem}
\label{smul}(Smul'jan's Theorem \cite{Smu}) Let $A,B,C$ be matrices of complex numbers, with $A$ and $C$ square matrices.\ Then
\begin{equation*}
\tilde A :=\left(
\begin{array}{cc}
A & B \\
B^{\ast } & C%
\end{array}%
\right) \geq 0 \iff \left\{
\begin{array}{c}
A\geq 0 \\
B=AW \; (\textrm{for some~} W)\\
C\geq W^{\ast }AW%
\end{array}%
\right. .
\end{equation*}
Moreover, $\rank \tilde A=\rank A \iff C=W^{\ast }AW.$\ 
\end{theorem}

\begin{remark}
When the extension $\tilde A$ in Theorem \ref{smul} has the same rank as $A$, we say that $\tilde A$ is a flat extension of $A$.\ 
Besides satisfying this theorem, an extension $\mathcal M(d+1)$ must maintain the moment matrix structure, that is, the $C$-block must be  Hankel. \ This condition makes generating flat extensions quite difficult in many instances. \  
\end{remark}

We now discuss how we can find an explicit formula for a representing measure. \  
Suppose $\mathcal{M}(d)$ admits a  positive  extension $\mathcal M(d+k)$ for some $k\in \mathbb Z_+$ that has a flat extension  $\mathcal M(d+k+1)$. \ Thus, $\beta$ has a $\opn{rank}\; \mathcal{M}(d+k)$-atomic measure $\mu$ and let $r:=\opn{rank}\; \mathcal{M}(d+k)$. \ 
The flat extension theorem says that the algebraic variety $\av$ of $\mcal M(d+k+1)$  consists of exactly $r$ points and we may write $\av=\q (x_1,y_1),\ldots,(x_r,y_r)\w$. \ 
Denote the Vandermonde matrix $V$ as
\begin{equation}
V=\begin{pmatrix}
1&x_1&y_1& x_1^2&x_1 y_1 & y_1^2& \cdots & x_1^{d+k}& \cdots & y_1^{d+k}\\
\vdots&\vdots&\vdots&\vdots&\vdots&\vdots&\vdots&\vdots&\vdots&\vdots\\
1&x_r&y_r& x_r^2&x_r y_r & y_r^2& \cdots & x_r^{d+k}& \cdots & y_r^{d+k}\\
\end{pmatrix}
\end{equation}
If $\mcal B:=\q \tb{t}_1,\ldots, \tb{t}_r  \w$ is the basis for the column space of $\mcal M(d+k)$ and if $V_{\mcal B}$ is the submatrix of $V$ with columns labeled as in $\mcal B$, then we can find the densities by solving: 
\begin{equation}
V_{\mcal B}^T 
\begin{pmatrix}
\rho_1 & \rho_2 &\cdots & \rho_r
\end{pmatrix}^T
=\begin{pmatrix}
\Lambda(\tb{t}_1)&\Lambda(\tb{t}_2)&\cdots& \Lambda(\tb{t}_r) 
\end{pmatrix}^T.\ 
\end{equation}
Finally, we have  $\mu=\sum_{k=1}^r \rho_k \delta_{(x_k,y_k)}.$\   

\subsection*{Degree-One Transformations}

We briefly review a tool that will allow us to convert the given moment problem into a simpler one; this tool is known as the invariance of moment problems under degree-one transformations. \ The complex version is provided in \cite{tcmp6}; we   adapt the notation in \cite{tcmp6} to obtain a real version. 

For $a,b,c,d,e,f\in \mathbb{R}$ with $bf\neq ce$, let $\Psi(x,y)\equiv (\Psi_1(x,y),\Psi_2(x,y)):=(a+bx+cy,d+ex+fy)$ for $x,y\in\re$. \  
If $\Lambda_\beta$ denotes the Riesz functional associated with $\beta$, then 
given $\beta \equiv \beta^{(2d)}$ we build a new (equivalent) moment sequence $\tilde\beta \equiv \tilde\beta^{(2d)} \equiv \q \tilde\beta_{ij}\w $ given by $\tilde\beta_{ij}:=\Lambda_\beta(\Psi_1^i \Psi_2^j)$ $(0\leq i+j\leq 2d)$. \ We immediately check that $\Lambda_{\tilde \beta}(p) =\Lambda_\beta(p\circ \Psi)$ for every $p\in \mcal P_d$. \ 

\begin{proposition}\cite[cf. Proposition 1.7]{tcmp6} \label{deg-one}
\textup{(Invariance under degree-one transformations)}\  Let $\mcal 
{M}(d)$ and $\tilde{\mcal {M}}(d)$ be the moment matrices associated with $\beta$ and
$\tilde{\bt}$, resp., and let $J\hat{p}:=\widehat{p\circ\Psi}$ \textup{(}$p\in%
\mathcal{P}_{d}$\textup{).}\  The following statements hold: 
\begin{enumerate}[(i)]
\item  \label{lininv(1)}$\tilde{{\mcal M}}(d)=J^T{\mcal M}(d)J$;
\item  \label{lininv(2)}$J$ is invertible;
\item  \label{lininv(3)}$\tilde{\mcal{M}}(d)\geq0\Leftrightarrow {\mcal M}(d)\geq0$;
\item  \label{lininv(4)}$\rank\tilde{\mcal{M}}(d)=\rank{\mcal M}(d)$;
\item  \label{lininv(5)}The formula $\mu=\tilde{\mu}\circ\Psi$ establishes a
one-to-one correspondence between the sets of representing measures for $\bt$ and $\tilde{\bt}$, which preserves measure class and cardinality
of the support; moreover, $\varphi(\supp\mu )=\supp\tilde{\mu}$;
\item  \label{lininv(6)}$\mcal{M}\left( d\right) $ admits a flat extension if and
only if $\tilde{\mcal{M}}\left( d\right) $ admits a flat extension. \  
\end{enumerate}
\end{proposition}

We will now apply Proposition \ref{deg-one} to a cubic real binary moment sequence $\beta\equiv \beta^{(3)}: \q \beta_{00}, \beta_{10}, \beta_{01}, 
\beta_{20}, \beta_{11}, \beta_{02}, \beta_{30}, \beta_{21}, \beta_{12}, \beta_{03}\w$ with $\beta_{00}>0$. \  
Our strategy is to enlarge $\beta^{(3)}$ to $\beta^{(4)}$ by adding new undetermined moments of degree $4$; this extended finite sequence has an associated moment matrix $\mathcal{M}(2)$.  \ As we will see in the sequel, it is enough to consider the case when $\mathcal{M}(2)$ is ``normalized," that is, $\mathcal{M}(1)$ is the identity matrix.  \ For, the case when $\mathcal{M}(1)$ is singular can be dealt with easily using the results in \cite{tcmp1} and \cite{tcmp2}. \ We thus assume that $\beta_{00}=1$ and that the principal $2\times 2$ and $3\times 3$ minors of $\mathcal{M}(1)$, $d_2$ and $d_3$, resp., are strictly positive. \ A calculation using \textit{Mathematica} \cite{Wol} reveals that
\begin{eqnarray*}
d_2 &=&-\beta_{10}^2+\beta_{20}\\
d_3&=& -\beta_{02} \beta_{10}^2+2 \beta_{01} \beta_{10} \beta_{11}-\beta_{11}^2-\beta_{01}^2 \beta_{20}+\beta_{02} \beta_{20}. \ 
\end{eqnarray*}
Consider now the degree-one transformation
$$
\Psi(x,y) \equiv \left( a+bx + cy, d+ex+fy\right),
$$
where $ a:=\frac{\beta_{01}\beta_{20}-\beta_{10} \beta_{11}}{\sqrt{d_2 d_3}}$, $b:=\frac{\beta_{11}-\beta_{01} \beta_{10}}{\sqrt{d_2 d_3}}$, $c:=
- \sqrt{\frac{d_2}{d_3}}$ , $d:=-\frac{\beta_{10}}{\sqrt {d_2} }$, $e:= \frac{1}{\sqrt {d_2} }$, and $f:=0$. \  
Observe that
$$
bf-ce=- \sqrt{\frac{1}{d_3}}\neq0. \ 
$$
Through this transformation, and using \cite[Proposition 1.7]{tcmp6}, any positive semidefinite $\mcal M(2)$ with a nonsingular $\mcal M(1)$ can be translated to
\begin{equation}\label{e-nm1}
\begin{pmatrix}
1 & 0& 0 & 1 & 0 & 1 \\
0& 1& 0 & \tilde{\beta}_{30} & \tilde{\beta}_{21}& \tilde{\beta}_{12} \\
0 & 0 & 1 & \tilde{\beta}_{21} & \tilde{\beta}_{12} & \tilde{\beta}_{03} \\
 1 & \tilde{\beta}_{30} & \tilde{\beta}_{21} & \tilde{\beta}_{40} & \tilde{\beta}_{31} & \tilde{\beta}_{22} \\
0 & \tilde{\beta}_{21} & \tilde{\beta}_{12} & \tilde{\beta}_{31} & \tilde{\beta}_{22} & \tilde{\beta}_{13} \\
1 & \tilde{\beta}_{12} & \tilde{\beta}_{03} & \tilde{\beta}_{22} & \tilde{\beta}_{13} & \tilde{\beta}_{04}
\end{pmatrix} 
=: \mathcal{M}[a_0,a_1,a_2,a_3]
\end{equation}
where $a_i:=\tilde{\beta}_{3-i,i}$.

\subsection*{Recursively Determinate Moment Problems.} 

Our approach to the nonsingular cubic moment problem will require a key result from the theory of recursively determinate moment matrices \cite{tcmp13}, which we now briefly describe. \ We first recall that a moment matrix ${\mcal M}(d)$ is \textit{recursively determinate} if there are column dependence relations in ${\mcal M}(d)$ of the form
\begin{eqnarray} 
&&X^n = p(X,Y) ~~~ (p\in \mathcal P_{n-1}); \label{RDMP1}\\
&&Y^m = q(X,Y) ~~~(q\in \mathcal P_m, \; m \leq n, \; \textrm{and } q \textrm{ has no $y^m$ term}), \label{RDMP2}
\end{eqnarray}
where $\mathcal P_k$ denotes the subspace of polynomials in $\mathbb R[x,y]$ whose degree is less than or equal to $k$. \  
One of the main results in \cite{tcmp13} follows.

\begin{lemma} \ (\cite[Corollary 2.4]{tcmp13}, with $d=n=m=2$, so that $d=n+m-2$.) \label{lemmanew} 
Assume that $\mcal M(2)$ is positive semidefinite and admits column relations of the form (\ref{RDMP1}) and (\ref{RDMP2}), with $n=m=2$. \ Then $\mathcal{M}(2)$ admits a flat extension $\mcal M(3)$.
\end{lemma}

We recall that, in general, the solubility of a quartic moment problem requires the variety condition. \ However, Lemma \ref{lemmanew} says that the variety condition is superfluous if a positive semidefinite $\mcal M(2)$ with invertible $\mathcal{M}(1)$ has only two column relations $X^2=p(X,Y)$ and $Y^2=q(X,Y)$, where $p$ and $q$ are linear polynomials. \ In such a case, $\mcal M(2)$ has a flat extension $\mathcal{M}(3)$, and therefore a $4$-atomic representing measure. \ It follows that the pair of equations $x^2=p(x,y)$ and $y^2=q(x,y)$ has exactly $4$ common real roots. \


\section{Cubic Binary Moment Problems}\label{s-cmp}

As we have indicated before, the nontrivial cases of the cubic binary moment problem arise when the submatrix  $\mcal M(1)$ of $\beta^{(3)}$ is nonsingular. \ Moreover, as noted in \cite{Kim14} the positive semidefiniteness of $\mcal M(1)$ is always a necessary condition for the existence of a representing measure. \ Thus, in the sequel we focus on cubic binary moment problems with $\mcal M(1)$ positive definite. \ When this happens we say that $\beta^{(3)}$ is a \ti{nonsingular} cubic binary moment sequence. \

\subsection*{Main Results}

Using the degree-one transformation introduced in Section \ref{s-prel}, if  $\beta^{(3)}$ is a nonsingular cubic binary moment sequence we may always assume, without loss of generality,  that $\beta^{(3)}: \q 1,0,0,1,0,1,a_0,a_1,a_2,a_3\w$ and we may write 
\begin{equation}
\mathcal M{(2)}:=\left(
\begin{array}{cccccc}
 1 & 0 & 0 & 1 & 0 & 1 \\
 0 & 1 & 0 & {a_0} & {a_1} & {a_2} \\
 0 & 0 & 1 & {a_1} & {a_2} & {a_3} \\
 1 & {a_0} & {a_1} & \beta_{40} & \beta_{31} & \beta_{22} \\
 0 & {a_1} & {a_2} & \beta_{31} & \beta_{22} & \beta_{13} \\
 1 & {a_2} & {a_3} & \beta_{22} & \beta_{13} & \beta_{04}
\end{array}
\right),
\end{equation}
where $ \beta_{40}, \beta_{31}, \beta_{22}, \beta_{13}$ and $ \beta_{04}$ are undetermined new moments. \  
We will prove that the extended $\beta^{(4)}$ obtained from $\beta^{(3)}$ by adding the quartic moments $\beta_{40}$, $\beta_{31}$, $\beta_{22}$, 
$\beta_{13}$ and $\beta_{04}$ admits a representing measure, for appropriate choices of the new moments; as a result, $\beta^{(3)}$ also admits a representing measure $\mu$. \ The smallest cardinality of $\opn{supp} \mu$ will be $3$ in some cases, and $4$ in others.

Using Theorem \ref{smul}, we first determine under what conditions the extended matrix $\mathcal{M}(2)$ will be a flat extension of $\mathcal{M}(1)$. \ First, to ensure the positive semidefiniteness of $\mathcal{M}(2)$, and if we let $W:=B(2)$ (the upper right-hand $3\times3$ block of $\mathcal{M}(2)$), we see that $C(2)$ (the lower right-hand $3\times3$ block of $\mathcal{M}(2)$) must satisfy the inequality $C(2) \ge W^T \mcal M(1) W$, with equality characterizing flatness. \ Now,  
\begin{equation}\label{e-C2}
W^T (\mcal M(1))^{-1} W=W^TW= 
\begin{pmatrix}
1+a_0^2+a_1^2 &	a_0 a_1+a_1 a_2&	1+a_0 a_2+a_1 a_3\\
a_0 a_1+a_1 a_2&	a_1^2+a_2^2	& a_1 a_2+a_2 a_3\\
1+a_0 a_2+a_1 a_3&	a_1 a_2+a_2 a_3	&1+a_2^2+a_3^2
\end{pmatrix}.
\end{equation}
Consequently, $\mathcal M{(2)}$ is a flat extension of $\mcal M{(1)}$ if and only if 
\begin{eqnarray}
\beta_{40} &=& 1+a_0^2+a_1^2 \label{eq33} \\ 
\beta_{31} &=& a_0 a_1+a_1 a_2 \label{eq34} \\
\beta_{13} &=& a_1 a_2+a_2 a_3 \label{eq35} \\
\beta_{04} &=& 1+a_2^2+a_3^2, \; \textrm{and} \label{eq36} \\
k &:=& (1+a_0 a_2+a_1 a_3) - (a_1^2+a_2^2)=0. \label{eq37}
\end{eqnarray}
Condition (\ref{eq37}) is equivalent to the commutativity of the matrices defined in \cite{Kim14}. \ 

We are now ready to prove our first result.
 
\begin{theorem} \label{zeromain}
Let $\beta^{(3)}$ be a nonsingular cubic binary moment sequence, let $k$ be as in (\ref{eq37}) and assume that $k=0$. \ Then $\beta^{(3)}$ admits a $3$-atomic representing measure.
\end{theorem}

\begin{proof} From the discussion preceding the statement of Theorem \ref{zeromain}, the new quartic moments $\beta_{40}$, $\beta_{31}$, $\beta_{13}$ and $\beta_{04}$ must be defined using (\ref{eq33}) -- (\ref{eq37}), resp. \ As for $\beta_{22}$, we must use $1+a_0 a_2+a_1 a_3$, which in this case equals $a_1^2+a_2^2$, because $k=0$. \ With these definitions, we easily conclude that $\mathcal{M}(2)$ is a flat moment matrix extension of $\mathcal{M}(1)$, which gives the desired result.
\end{proof}

When $k\neq 0$, it is not possible to select new quartic moments so that $\mcal M(2)$ is a flat extension of $\mathcal{M}(1)$. \ Therefore, any positive semidefinite moment matrix extension $\mathcal{M}(2)$ will satisfy $\rank \mcal M(2)\geq 4$. \ Nevertheless, the following theorem shows that it is always possible to choose a set of quartic moments such that $\rank \mcal M(2) =4$. \ Once those moments have been appropriately chosen, the extended moment matrix $\mathcal{M}(2)$ will admit a flat extension $\mathcal{M}(3)$, and therefore a $4$-atomic representing measure for $\beta^{(4)}$, which is also a representing measure for the initial data sequence $\beta^{(3)}$.

\begin{theorem} \label{firstmain}
Let $\beta^{(3)}$ be a nonsingular cubic binary moment sequence, let $k$ be as in (\ref{eq37}) and assume that $k \ne 0$. \ Then $\beta^{(3)}$ admits a $4$-atomic representing measure.
\end{theorem}

\begin{proof} We will divide the proof into two cases: $k>0$ and $k<0$.

\textbf{Case 1} ($k>0$). \ As in the Proof of Theorem \ref{zeromain}, let $\beta_{40}$, $\beta_{31}$, $\beta_{13}$ and $\beta_{04}$ be given by (\ref{eq33}) -- (\ref{eq37}), resp. \ Since $k>0$, the positivity of $\mathcal{M}(2)$ will be preserved if we let $\beta_{22}:=1 + a_1 a_3 + a_2 a_4$. \ With this choice of $\beta_{22}$, the proposed extended matrix $\mathcal{M}(2)$ will be a positive semidefinite moment matrix, and such that the block $C(2)$ differs from $W^T (\mathcal{M}(1))^{-1}W$ in just the $(2,2)$-entry. \ As a result, $\rank \mcal M(2)=4$. \ A simple calculation now reveals that 
\begin{eqnarray*}
X^2 &=& \ti 1+a_0 X + a_1 Y \\
Y^2 &=& \ti 1 + a_2 X + a_3 Y.
\end{eqnarray*}
We now know that $\mathcal{M}(2)$ is positive semidefinite and recursively determinate, and by Lemma \ref{lemmanew}, $\mathcal{M}(2)$ admits a $4$-atomic representing measure; it follows that $\beta^{(3)}$ also admits a $4$-atomic representing measure.

\textbf{Case 2} ($k<0$). \ Here our strategy is to allow the rank to increase as we transition from $\mathcal{M}(1)$ to the compression of $\mathcal{M}(2)$ to the first $4$ rows and columns. \ This requires making the column $X^2$ \emph{linearly independent} of the columns $\emph{1}$, $X,$ and $Y$ in $\mathcal{M}(2)$. \ It is straightforward to observe that this can be easily accomplished by letting 
$$
\beta_{40} := 2 + a_1^2 +a_2^2.
$$
With this definition in hand, we now postulate that $\mathcal{M}(2)$ be a flat extension of its compression to the first $4$ rows and columns. \ A calculation using \emph{Mathematica} reveals that one can accomplish this by defining three of the remaining quartic moments as follows:
\begin{eqnarray*}
\beta_{31} &:= &a_1 a_2 + a_2 a_3,\\
\beta_{22} &:=& a_2^2  +a_3^2 ,\\
 \beta_{13} &:=& a_2 a_3 + a_3 a_4.
\end{eqnarray*}
Having chosen these moments, we now use Theorem \ref{smul} to determine the remaining quartic moment, $\beta_{04}$. \ Since we wish to make $\mathcal{M}(2)$ a flat extension of its above mentioned compression, a calculation using \emph{Mathematica} immediately yields
\begin{eqnarray} \label{b04}
\beta_{04} &=& 2+a_1^4+2 a_0 a_2+a_0^2 a_2^2+2 a_1^2 a_2^2+a_2^4+2 a_1 a_3+2 a_0 a_1 a_2 a_3+a_3^2+a_1^2 a_3^2 \nonumber \\
&& -2 a_1^2-2 a_0 a_1^2 a_2-a_2^2-2 a_0 a_2^3-2 a_1^3 a_3-2 a_1 a_2^2 a_3.
\end{eqnarray}
As a result, in $\mathcal{M}(2)$ we now have
\begin{equation}
XY=a_1X+a_2Y \label{eqxy}
\end{equation}
and
\begin{equation}
Y^2=p_1 1+p_2 X+p_3 Y+p_4 X^2 \label{eqy2}
\end{equation}
for suitable real scalars $p_1,p_2,p_3,p_4$ (which depend upon $a_0,a_1,a_2,a_3$); moreover, $p_4=-k$. \ We will now build a flat moment matrix extension $\mathcal{M}(3)$ of $\mathcal{M}(2)$. \ This will prove that $\mathcal{M}(2)$ admits a $4$-atomic representing measure; a fortiori, $\beta^{(3)}$ also admits a $4$-atomic representing measure, just as in Case 1 above.

To define $\mathcal{M}(3)$, we aim to preserve the property (RG). \ First, we observe that the columns $1$, $X$ and $Y$ in $\mathcal{M}(3)$ are obtained from those columns in $\mathcal{M}(2)$ by adding suitable cubic and quartic moments. \ Moreover, the columns $XY$ and $Y^2$ are defined using (\ref{eqxy}) and (\ref{eqy2}), while the columns $X^2Y$, $XY^2$ and $Y^3$ are obtained, via the functional calculus, from (\ref{eqxy}) and (\ref{eqy2}). \ For instance, 
$$
X^2Y:=a_1X^2+a_2XY=a_1X^2+a_2(a_1X+a_2Y)=a_2a_1X+a_2^2Y+a_1X^2
$$
and
\begin{eqnarray*}
Y^3&:=&p_1 Y+p_2 XY+p_3 Y^2+p_4 X^2Y \\
&=&p_1 Y+p_2 (a_1X+a_2Y)+p_3 (p_1 1+p_2 X+p_3 Y+p_4 X^2) \\
&&+p_4 (a_2a_1X+a_2^2Y+a_1X^2).
\end{eqnarray*}
It follows that both $X^2Y$ and $Y^3$ are linear combinations of the columns $1$, $X$, $Y$ and $X^2$. \ Now, to define $XY^2$ one can use either (\ref{eqxy}) or (\ref{eqy2}).\ However, property (RG) requires that both definitions of $XY^2$ be compatible. \ In other words, the expressions
\begin{eqnarray*}
XY^2 &\equiv& a_1XY+a_2Y^2=a_2p_1 1+(a_1^2+a_2p_2)X+(a_1a_2+a_2p_3)Y+a_2p_4X^2 \\ 
&& (\textrm{obtained using (\ref{eqxy})})
\end{eqnarray*}
and
\begin{eqnarray*}
XY^2 &\equiv& p_1X+p_2X^2+p_3XY+p_4X^3=(p_1+a_1p_3)X+a_2p_3Y+p_2X^2+p_4X^3 \\
&& (\textrm{obtained using (\ref{eqy2})})
\end{eqnarray*}
must be identical. \ Since $p_4=-k\ne0$, we immediately get
\begin{equation} \label{eqx3}
X^3=\frac{1}{p_4}[a_2p_1 1+(a_1^2+a_2p_2-p_1-a_1p_3)X+a_1a_2Y+(a_2p_4-p_2)X^2],  
\end{equation}
which we can then use to define the column $X^3$. \ Close examination of (\ref{eqx3}) at the level of the fourth row in $\mathcal{M}(3)$ leads to a formula for the quintic moment $\beta_{50}$. \ This value must then be inserted in the seventh row of $X^2$, to complete the definition of $X^2$ in $\mathcal{M}(3)$. As a result, in the new moment matrix $\mathcal{M}(3)$ we have exhibited each cubic column as a linear combination of columns associated with monomials of degree at most $2$. \ This means that $\mathcal{M}(3)$ is a flat extension of $\mathcal{M}(2)$, as desired. \ The proof is now complete.
\end{proof} 

\begin{remark}
The quartic moment $\beta_{04}$ defined by (\ref{b04}), is nonnegative, being a diagonal entry of the positive semidefinite matrix $\mathcal{M}(2)$. \ One can say more, however, by appealing to the theory of Semidefinite Programming (SDP). \ As is well known, a polynomial $f\in \mcal P_{2d}$ is a sum of squares if and only if $f=\tb z^T Q \tb z$ for some square matrix $Q\geq 0$, where $\tb z$ is the vector of monomials of degree less than or equal to $d$. \  If we let $f \equiv f(a_0,a_1,a_2,a_3):=\beta_{04}-1$ and $\tb y:= (1,a_2,a_3,a_1^2,a_2^2,a_0 a_2,a_1 a_3)$, a calculation using \emph{Mathematica} reveals that $f \ge 0$ if and only if $\tb y^T R \tb y \ge 0$, where 
\begin{equation}
R:=  
\left(
\begin{array}{ccccccc}
 1 & 0 & 0 & -1 & -1 & 1 & 1 \\
 0 & 1 & 0 & 0 & 0 & 0 & 0 \\
 0 & 0 & 1 & 0 & 0 & 0 & 0 \\
 -1 & 0 & 0 & 1 & 1 & -1 & -1 \\
 -1 & 0 & 0 & 1 & 1 & -1 & -1 \\
 1 & 0 & 0 & -1 & -1 & 1 & 1 \\
 1 & 0 & 0 & -1 & -1 & 1 & 1
\end{array}
\right).
\end{equation}
Since $R$ is a flat extension of its $3\times3$ compression to the first three rows and columns, it is clear that $R \ge 0$. \ It follow that $\beta_{04} \equiv f+1 \ge 1 > 0$.
\end{remark}

\textit{Acknowledgment}.  \ Some of the proofs in this paper were obtained using calculations with the software tool \textit{Mathematica} \cite{Wol}. \

\bibliographystyle{plain}

\end{document}